\documentclass[12pt]{article}
\usepackage{fullpage}
\usepackage[utf8]{inputenc}

\usepackage[square,sort,comma,numbers]{natbib}
\usepackage{graphicx}
\usepackage{amscd,amsmath,amsthm,amssymb}
\usepackage{verbatim}
\usepackage[all]{xy}
\usepackage{color}
\usepackage{hyperref}
\usepackage{tikz, float} \usetikzlibrary {positioning}
\usepackage{MnSymbol}
\usetikzlibrary{calc}

\newcommand{\BLambda}{{B\Lambda}}
\newcommand{\id}{\text{id}}
\newcommand{\ini}{\text{in}}
\newcommand{\I}{\mathcal{I}}

\def\frk{\mathfrak}               

\def\Phi{{\frk n}}
\def\Phi{{\frk N}}

\def\x{{\mathbf x}}

\def\A{{\mathcal A}}

\def\I{{\mathcal I}}
\def\opn#1#2{\def#1{\operatorname{#2}}} 
\opn\chara{char} \opn\length{\ell} \opn\pd{pd} \opn\rk{rk}
\opn\projdim{proj\,dim} \opn\injdim{inj\,dim} \opn\rank{rank}
\opn\depth{depth} \opn\grade{grade} \opn\height{height}
\opn\embdim{emb\,dim} \opn\codim{codim}

\opn\Tr{Tr} \opn\bigrank{big\,rank}
\opn\superheight{superheight}\opn\lcm{lcm}
\opn\trdeg{tr\,deg}
\opn\reg{reg} \opn\lreg{lreg} \opn\ini{in} \opn\lpd{lpd}
\opn\size{size} \opn\sdepth{sdepth}
\opn\link{link}\opn\fdepth{fdepth}\opn\lex{lex}
\opn\LM{LM}
\opn\LC{LC}
\opn\NF{NF}
\opn\Merge{Merge}
\opn\sgn{sgn}

\opn\div{div} \opn\Div{Div} \opn\cl{cl} \opn\Pic{Pic}
\opn\Prin{Prin}
\opn\op{op}
\opn\indeg{indeg} \opn\outdeg{outdeg}
\opn\red{red}
\opn\Spec{Spec} \opn\Supp{Supp} \opn\supp{supp} \opn\Sing{Sing}
\opn\Ass{Ass} \opn\Min{Min}\opn\Mon{Mon} \opn\val{val}
\opn\Ann{Ann} \opn\Rad{Rad} \opn\Soc{Soc}
 \opn\Ker{Ker} \opn\Coker{Coker} \opn\Am{Am}
\opn\Hom{Hom} \opn\Tor{Tor} \opn\Ext{Ext} \opn\End{End}
\opn\Aut{Aut} \opn\id{id}

\opn\nat{nat}
\opn\pff{pf}
\opn\Pf{Pf} \opn\GL{GL} \opn\SL{SL} \opn\mod{mod} \opn\ord{ord}
\opn\Gin{Gin} \opn\Hilb{Hilb}\opn\sort{sort}
\opn\span{span}
\opn\Image{Image}
\opn\vol{Vol}
\opn\aff{aff} \opn\con{conv} \opn\relint{relint} \opn\st{st}
\opn\lk{lk} \opn\cn{cn} \opn\core{core} \opn\vol{vol}
\opn\link{link} \opn\star{star}\opn\lex{lex}\opn\set{set}
\opn\dist{dist}
\opn\gr{gr}

\def\pot#1#2{#1[\kern-0.28ex[#2]\kern-0.28ex]}

\opn\dirlim{\underrightarrow{\lim}}
\opn\inivlim{\underleftarrow{\lim}}
\let\to=\rightarrow

\def\Implies{\ifmmode\Longrightarrow \else
        \unskip${}\Longrightarrow{}$\ignorespaces\fi}
\def\implies{\ifmmode\Rightarrow \else
        \unskip${}\Rightarrow{}$\ignorespaces\fi}
\def\iff{\ifmmode\Longleftrightarrow \else
        \unskip${}\Longleftrightarrow{}$\ignorespaces\fi}

\let\:=\colon
\newtheorem{Theorem}{Theorem}[section]
\newtheorem{Lemma}[Theorem]{Lemma}
\newtheorem{Corollary}[Theorem]{Corollary}
\newtheorem{Proposition}[Theorem]{Proposition}

\theoremstyle{remark}
\newtheorem{Remark}[Theorem]{Remark}

\theoremstyle{definition}
\newtheorem{Example}[Theorem]{Example}

\newtheorem{Definition}[Theorem]{Definition}

\DeclareMathOperator{\Gr}{Gr}

\let\kappa=\varkappa
\def\qed{\ifhmode\textqed\fi
      \ifmmode\ifinner\quad\qedsymbol\else\dispqed\fi\fi}
\def\textqed{\unskip\nobreak\penalty50
       \hskip2em\hbox{}\nobreak\hfil\qedsymbol
       \parfillskip=0pt \finalhyphendemerits=0}
\def\dispqed{\rlap{\qquad\qedsymbol}}

\opn\dis{dis}
\def\pnt{{\raise0.5mm\hbox{\large\bf.}}}

\opn\Lex{Lex}
\opn\syz{{\rm syz}}
\opn\spoly{{\rm spoly}}
\opn\LM{{\rm LM}}
\opn\lm{{\rm lm}}
\opn\lcm{{\rm lcm}} \opn\A{\mathcal A}

\numberwithin{equation}{section}

\begin{document}

\title{Toric degenerations of Grassmannians \\ from matching fields}
\author{Fatemeh Mohammadi and Kristin Shaw}

\maketitle

\tikzstyle{Cgray}=[scale = .45,circle, fill = gray, minimum size=3mm] \tikzstyle{Cblack}=[scale = .7,circle, fill = black, minimum size=3mm]
\tikzstyle{Cblue}=[scale = .5,circle, fill = blue, inner sep = 0pt, minimum
size=3mm]
\tikzstyle{C1}=[scale = .7,circle, fill = black!0, inner sep = 0pt, minimum
size=3mm]

\tikzstyle{test2}=[scale = 1.5,circle, fill = black!0, inner sep = 0pt, minimum
size=3mm]
\tikzstyle{Cwhite}=[scale = .45,circle, fill = white, minimum size=3mm] 
\tikzstyle{Cblack2}=[scale = .3,circle, fill = black, minimum size=3mm] 
\tikzstyle{Cblack}=[scale = .3,circle, fill = black, minimum size=3mm]
\tikzstyle{C0}=[scale = .9,circle, fill = black!0, inner sep = 0pt, minimum size=3mm]
\tikzstyle{C1}=[scale = .7,circle, fill = black!0, inner sep = 0pt, minimum size=3mm]
\tikzstyle{Cred}=[scale = .4,circle, fill = red, minimum size=3mm] 
\tikzstyle{Cblack3}=[scale = .4,circle, fill = black, minimum size=3mm]

\abstract{
We study the algebraic combinatorics of monomial degenerations of Pl\"ucker forms which is governed by matching fields in the sense of Sturmfels and Zelevinsky. 
We provide a necessary condition for a matching field to yield a Khovanskii basis of the Pl\"ucker algebra for $3$-planes in $n$-space. When the ideal associated to the matching field is quadratically generated this condition is both necessary and sufficient. 
Finally, we describe a family of matching fields, called $2$-block diagonal, whose ideals are quadratically generated. These matching fields produce 
a new family of  toric degenerations of $\Gr(3, n)$.}

\section{Introduction}
In this note we offer a new family of  toric degenerations 
of $\Gr(3, n)$ arising from monomial degenerations of the Pl\"ucker forms. 
Toric degenerations provide a useful tool to study algebraic varieties. This is mainly  because toric geometry is inextricably linked to the theory of polytopes and polyhedral fans. 
Combinatorial invariants of polytopes provide geometric information about toric varieties, and many of these invariants 
are preserved under degeneration. Here, a toric degeneration is a Gr\"obner degeneration such that the corresponding initial ideal is binomial and prime, see Definition \ref{def:toric}.

For general Grassmannians  and flag varieties there  are prototypic examples of toric degenerations which are related  to  young tableaux, Gelfand-Cetlin integrable systems, and their polytopes \cite{ FvectorGC, KOGAN}.
In the case of the   Grassmannian $\Gr(2,n)$,  there are many other toric degenerations generalising this primary example. Namely, any  trivalent tree with $n$ number of labelled leaves gives rise to a toric degeneration of $\Gr(2, n)$.
The toric variety is governed by the isomorphism type of the trivalent tree \cite{SpeyerSturmfels, Witaszek}.
These degenerations are related to bending systems on polygon spaces and  integrable systems \cite{Kapovich, NISHINOU}.

The Gelfand-Cetlin degenerations arise from  monomial initial degenerations of the Pl\"ucker forms. In fact, these degenerations arise from the theory of Khovanskii bases, which are also called SAGBI bases,  see Definition \ref{def:Khovanskii}. 
The leading term of each Pl\"ucker form in this case  is the monomial of the determinant corresponding to the identity permutation.  The new degenerations of $\Gr(3, n)$ provided here  depend only on the underlying \emph{coherent matching field}. A matching field is a choice of permutation for each  Pl\"ucker form, see Definitions \ref{def:matching} and \ref{def:coherent}. A coherent matching field provides a monomial degeneration of the Pl\"ucker forms and are therefore candidates for Khovanskii bases.

To a general matching field, we associate a toric ideal in the polynomial ring $\mathbb{K}[x_{ij}]$ where $\mathbb{K}$ is a field.  These matching field ideals are most conveniently represented by matching field tableaux, which we also introduce here. These tableaux generalize Young tableaux which are usually strictly increasing in the columns. Columns of a matching field tableau are filled according to the permutation chosen by the matching field. 
A matching field ideal is the kernel of a monomial map, and hence toric, see Equation $(2.2)$.  Moreover, it is generated by binomials which come from  pairs of matching field tableaux whose contents are row-wise equal. 
The Pl\"ucker forms are a Khovanskii basis of the Pl\"ucker algebra with respect to a matching field if and only if its matching field ideal is equal to the initial ideal of the corresponding degeneration of the Pl\"ucker ideal, see Theorem \ref{Theorem:Stu} or \cite[Theorem  11.4]{sturmfels1996grobner}. Therefore, obtaining a Khovanskii basis from a matching field is equivalent  to obtaining a toric degeneration of the Grassmannian.  
 
From a  weight matrix  that produces  a monomial degeneration of the Pl\"ucker forms, we can produce a tropical hyperplane arrangement, and the matching field can be determined from this geometric picture  \cite{fink2015stiefel}.
Using the  associated tropical hyperplane arrangement, we introduce the notion of hexagonal matching fields of size $3 \times 6$ and non-hexagonal matching fields, see Definition \ref{def:hex}. This leads to our first theorem. 

\begin{Theorem}\label{Theorem:necess}
If a $3 \times n$ matching field produces a toric degeneration of $\Gr(3,n)$, then it is non-hexgonal. 
\end{Theorem}

We also define  submatching fields by using natural maps between Grassmannians of different sizes, see Definition \ref{def:submatching}.
This allows us to extend Theorem \ref{Theorem:necess} to higher Grassmannians. 
\begin{Theorem}\label{Theorem:subhex}
A $k \times n$ matching field that has a hexagonal submatching field does not produce a toric degeneration of the Grassmannian $\Gr(k, n)$. 
\end{Theorem}

If a $3 \times n$ matching field ideal is quadratically generated, then the necessary  condition from Theorem \ref{Theorem:necess} is also sufficient.   

\begin{Theorem}\label{Theorem:iff}
A $3 \times n$ matching field whose ideal is quadratically generated provides a toric degeneration of $\Gr(3, n)$ if and only if it is non-hexagonal. 
\end{Theorem}

Describing a generating set of toric ideals is a well-studied and difficult problem. In particular, proving that an ideal is quadratically generated is quite a difficult task. There are some combinatorial criteria for the toric ideals arising from graphs, matroids and simplicial complexes to be generated by quadratics, see e.g. \cite{White, Ohsugi, Cone, Mateusz}. Such a criterion guarantees that the associated Koszul algebra is normal. 

Not all coherent matching field ideals are quadratically generated. The first examples that we know of are of size $3 \times 8$. 
However, we introduce a class of matching fields of size $3 \times n$, called \emph{block diagonal}, which are quadratically generated when they have $2$ blocks.

\begin{Theorem}\label{Theorem:2block}
The ideal of a $2$-block diagonal matching field  of size $3 \times n$  is quadratically generated. 
\end{Theorem}

\begin{Corollary}\label{cor:blockdiag}
A $2$-block diagonal matching field produces a toric degeneration of the Grassmannian $\Gr(3, n)$. 
Equivalently, when $k= 3$ the Pl\"ucker forms are a Khovanskii basis with respect to any weight matrix arising from a  $2$-block diagonal matching field.  
\end{Corollary}

Before reviewing the contents of the paper we would like to comment on some related works.  Rietsch and  Williams describe families of toric degenerations of Grassmannians arising from  plabic graphs \cite{RietschWilliams}. In \cite{Bossinger}, Bossinger et.~al.~show that already in the case of $\Gr(3, 6)$ there is a discrepancy between the toric degenerations arising on one hand from plabic graphs and on the other hand from top dimensional cones of the tropical Grassmannian. The toric degenerations studied here are a subset of the latter type coming from top dimensional cones of the tropical Grassmannian associated to Stiefel tropical linear spaces, see \cite{fink2015stiefel} for more details. In general, the combinatorial connection between matching fields and plabic graphs is still unknown. 
Kaveh and Manon provide a general connection between tropical geometry and Khovanskii bases (and hence toric degenerations) in \cite{Kaveh}. This question has been studied in \cite{bossinger2017computing} for small flag varieties. Here, we are interested in determining  when the Pl\"ucker  forms are a Khovanskii basis, equivalently when the associated initial degeneration of the Pl\"ucker ideal is toric. 
The results presented here offer a family of examples 
that fit into the  general framework of Kaveh and Manon and which are linked to combinatorics. Lastly, we would like to remark that toric degenerations of flag varieties and Schubert varieties arising from matching fields is another open direction of research at the present time. 

We finish the introduction with an outline of the paper. Section \ref{prelim} fixes notations for the Grassmannians and introduces matching fields. In Section \ref{sec:arr}, we review  the connection between matching fields and tropical hyperplane arrangements. Here we introduce the notion of hexagonal matching fields and prove Theorems \ref{Theorem:necess},  \ref{Theorem:subhex}, and 
\ref{Theorem:iff}. Block diagonal matching fields are introduced in Section \ref{sec:block} and the proof of Theorem \ref{Theorem:2block} is also given here. The final section defines matching field polytopes and provides some examples as well as  remarks about their combinatorics. 

\vspace{0.5cm}
\noindent{\bf Acknowledgement.} {The authors are very grateful to Alex Fink, J\"urgen Herzog, Mateusz Micha\l{}ek, Felipe Rinc\'on and Bernd Sturmfels
for 
many helpful conversations. We would also like to thank  Georg Loho and Volkmar Welker for helpful comments on a preliminary version of this manuscript.

This project was started while the second author was visiting MPI Leipzig and completed with the support of the Bergen Research Foundation project  grant ``Algebraic and topological cycles in tropical and complex geometry". The first author was partially
supported by EPSRC grant EP/R023379/1. 
}

\smallskip
We would like to remark that this paper was accepted for a poster in FPSAC 2018. However the extended abstract did not appear in the proceedings because the authors could not attend the conference.

\section{Preliminaries}\label{prelim} 

Throughout we set  $[n]:= \{1, \dots , n\}$ and we use  $\mathbf{I}_{k, n}$ to denote the collection of subsets of $[n]$ of size $k$. The symmetric group on $k$ elements is denoted by $S_k$. We also fix a field $\mathbb{K}$. 

\smallskip

The Grassmannian $\Gr(k, n)$ is the space of all $k$ dimensional linear subspaces of $\mathbb{K}^n$. A point in $\Gr(k, n)$ can be represented by a $k\times n$ matrix with entries in $\mathbb{K}$. 
Let $X=(x_{ij})$ be a $k\times n$ matrix of indeterminates. For a subset $I = \{i_1,\ldots,i_k\} \in \mathbf{I}_{k, n}$,  let $X_I$ denote the $k\times k$ submatrix with the column indices $i_1,\ldots,i_k$. 
The \emph{Pl\"ucker forms (or Pl\"ucker coordinates)} are $P_I = \text{det}(X_I)$  for $I \in\mathbf{I}_{k, n}$. These forms determine the \emph{Pl\"ucker embedding} from
$ {\rm Gr}(k,n)$ into $\mathbb{P}^{{n\choose k}-1}$. 

In the following, we consider the polynomial ring $\mathbb{K}[x_{ij}]$ on the variables $x_{ij}$ with $1\leq i\leq k$ and $1\leq j\leq n$ and the polynomial ring $\mathbb{K}[P_I]$ on the Pl\"ucker variables with $|I|=k$.

\begin{Definition}\label{def:ideal}The Pl\"ucker ideal $\I_{k,n}$ is defined as the kernel of the map 
\begin{eqnarray}\label{eqn:pluckermap}
\begin{split} 
\psi\ \colon\ & \mathbb{K}[P_I]  \rightarrow \mathbb{K}[x_{ij}]  
\\ 
& P_{I}   \mapsto \text{det}(X_I).
\end{split}
\end{eqnarray}
The \emph{Pl\"ucker algebra} is the finitely generated algebra $\mathbb{K}[P_I ]/{\mathcal{I}_{k,n}}$ 
denoted by $\mathcal{A}_{k,n}$.
\end{Definition}

\begin{Definition}\label{def:matching}
A $k\times  n$ matching field is a map $\Lambda: \mathbf{I}_{k,n}  \to S_k$.
\end{Definition}

Given a $k\times n$ matching field $\Lambda$ and a subset $I = \{i_1, \dots , i_k\}  \in \mathbf{I}_{k, n}$ 
we consider the set to be ordered by $i_1 < \dots < i_{k}$. We think of the permutation {$\sigma=\Lambda(I)$} as inducing a new ordering on the elements of $I$ 
where the position of $i_s$ is determined by the value of $\sigma(s)$.
It is convenient to represent the variable $P_I$ as a $k \times 1$ tableau 
 where $(\sigma(r), 1)$ contains $i_{r}$.

\begin{Definition}
Let $\Lambda$ be a size $k \times n$ matching field. 
A $\Lambda$-tableau is a tableau of size $k \times d$ for any $d\geq 1$ with entries in $[n]$,
so that the entries in each column are pairwise distinct and filled according to the order determined by $\Lambda$. 
\end{Definition}

\begin{Example}\label{ex:diagonal1}
The \emph{diagonal matching field} assigns to
each subset $I \in \mathbf{I}_{k,n}$ the identity permutation \cite[Example 1.3]{sturmfels1993maximal}.
Therefore, a $\Lambda$-tableau is a rectangular tableau of size $k \times d$ filled with entries in $[n]$ such that the columns are strictly increasing.  
\end{Example}

\begin{Example}\label{ex:pointed1}
A $k \times n$ matching field is called \emph{pointed} if there exists $i_1, \dots, i_k \in [n]$ such that if $i_s \in I$ for some 
$1 \leq s \leq k$ then $\Lambda(I)(s) = s$ \cite[Example 1.4]{sturmfels1993maximal}. In other words, if $i_s \in I$ for some $1\leq s \leq k$ then the column corresponding to $P_I$ contains $i_s$ in row $s$. Below are the tableaux representing $P_I$ for a pointed matching field of size $3 \times 5$ which is otherwise filled diagonally:
$$\begin{array}{c}1 \\2  \\ 3 \end{array} , \quad
\begin{array}{c} 1 \\2  \\4 \end{array},\quad
\begin{array}{c}1  \\2 \\ 5 \end{array} , \quad
\begin{array}{c}1 \\ 4 \\ 3  \end{array} ,\quad
\begin{array}{c}1 \\5  \\ 3\end{array} ,\quad
\begin{array}{c}4 \\ 2\\ 3  \end{array} ,\quad
\begin{array}{c}5 \\2 \\ 3\end{array} ,\quad
\begin{array}{c} 4 \\ 2   \\ 5 \end{array}. 
$$
Generalising this,  we say that a size $k \times n$ matching field $\Lambda$ is pointed on $S \subset [n]$ if for all $i \in S$ there exists a $j_i$ such that $i$ always appears in row $j_i$ of any  $\Lambda$ matching field tableau. 
Here $S$ need not have size equal to $k$. 
\end{Example}

A monomial $\Pi_{I \in A} P_I $ can be represented by a  $\Lambda$-tableau of size $k \times |A|$ given by the concatenation of the columns with content $I$ filled according to the matching field. 
To each monomial $\Pi_{I \in A} P_I $ we associate a sign $\epsilon_{A} = \pm 1$ determined by the signature of the permutations $\Lambda(I)$ for all $I \in A$. More precisely, 
$$\epsilon_{A} = \Pi_{I \in A} \text{sgn}(\Lambda(I)).$$

\begin{Definition}
Given a matching field $\Lambda$, the \emph{matching field  ideal} $\mathcal{I}_{\Lambda} \subset \mathbb{K}[x_{ij}]$ is  generated by the 
 binomial relations  
\begin{eqnarray}\label{eq:binomials}
\epsilon_A \Pi_{I \in A} P_I - \epsilon_B \Pi_{J \in B} P_J
\end{eqnarray}
if and only if the  contents of the corresponding  $\Lambda$-tableau of size $k \times |A|$  are  row-wise equal.
\end{Definition}

To $I \in \mathbf{I}_{k, n}$ {with $\sigma=\Lambda(I)$} we associate the monomial 
$$ 
\textbf{x}_{\Lambda(I)}:=x_{\sigma(1) i_{1}}x_{\sigma(2)i_2}\cdots x_{{\sigma(k)i_k}}. 
 $$
A $k\times n$ matching field $\Lambda$, 
gives a map of polynomial rings 
\begin{eqnarray}\label{eqn:monomialmap}
\begin{split}
\phi_{\Lambda} \colon\  & \mathbb{K}[P_I]  \rightarrow \mathbb{K}[x_{ij}]  
\\ 
& P_{I}   \mapsto \text{sgn}(\Lambda(I)) \textbf{x}_{\Lambda(I)}.
\end{split}
\end{eqnarray}

\smallskip

\begin{Proposition}
Given a matching field $\Lambda$, the matching field ideal $\mathcal{I}_{\Lambda}$ is the kernel of the monomial map
$\phi_{\Lambda}$ from Equation (\ref{eqn:monomialmap}). 
\end{Proposition}

\begin{Definition}\label{def:coherent}
A $k \times n$ matching field $\Lambda$  is \emph{coherent} if there exists a $k \times n$  matrix $M$ with entries in $\mathbb{R}$ 
such that 
for every  $I  \in \mathbf{I}_{k, n}$ 
the initial form of the  Pl\"ucker form  $P_I \in \mathbb{K}[x_{ij}]$, the sum of all terms in $M_I$ of lowest weight, is  $\text{in}_M (P_I) = \text{sgn}(\Lambda(I)) \mathbf{x}_{\Lambda(I)}$.  
In this case, we say that the matrix $M$ \emph{induces the matching field} $\Lambda$. 
\end{Definition}

\begin{Example}
When  $k = 2$ all coherent matching fields are induced by a total ordering on the set $[n]$, see \cite[Proposition 1.11]{sturmfels1993maximal}. 
\end{Example}

\begin{Example}\label{ex:diagonal2}
The diagonal matching field of size $k \times n$  is coherent 
\cite[Example 1.3]{sturmfels1993maximal}.
For example, this matching field is induced by a $k\times n$ weight matrix $M$ whose $i,j$-th entry 
is  $(i-1)(n-j)$. Therefore, we have,  
\[
M=\begin{bmatrix}
    0  & 0  & 0  & \cdots & 0  \\
    n-1 &\cdots & 2 & 1 & 0  \\
   2(n-1) &\cdots &   4 & 2  & 0    \\
    \cdots  & \cdots  & \cdots &  \cdots   \\
     (k-1)(n-1)& \cdots & 2(k-1)  & k-1   & 0    \\
\end{bmatrix}.
\]
For any size $k$ subset $I=\{i_1,i_2,\ldots,i_k\}$ with $i_1<i_2<\cdots<i_k$ the unique term in the determinant of $M_I$ with the lowest weight is the diagonal term.  
\end{Example}

\begin{Example}\label{ex:pointed2}
Recall the notion of a pointed matching field from Example \ref{ex:pointed1}. If a pointed matching field is coherent then it can be induced by a weight matrix $M$ such that for all $1 \leq s \leq k$ the  $i_s$-th column has entries $M >\!\!>0 $ except for in row $s$ where the entry can be chosen to be $0$. 
\end{Example}

\begin{Definition}\label{def:toric}
Let $\I$ be an ideal in the polynomial ring $S = \mathbb{K}[y_1, \dots, y_m]$ and let $w \in \mathbb{R}^m$. 
The initial degeneration with respect to $w$ is called \emph{toric} if the initial ideal $\ini_w(\I)$ is prime and binomial. 
\end{Definition}

\begin{Definition}\cite{robbiano1990subalgebra}\label{def:Khovanskii}
The set of Pl\"ucker forms $\{P_I\}_{I\in \mathbf{I}_{k,n}}\subset \mathbb{K}[x_{ij}]$ is a \emph{Khovanskii basis} for the Pl\"ucker algebra $\mathcal{A}_{k,n}$ with respect to a weight matrix $M$ 
if for all $I \in  \mathbf{I}_{k,n} $ the initial form $\ini_M(P_I) $ is a monomial and 
$\ini_{w_M}(\mathcal{A}_{k,n})=\mathbb{K}[\ini_M(P_I)]_{I \in \mathbf{I}_{k,n}}$. Here $w_M$ is the weight vector on the variables $P_I$ induced by the weight matrix $M\in\mathbb{R}^{k\times n}$ on the variables $x_{ij}$. 

\end{Definition}

The following theorem intimately relates Khovanskii bases and toric initial  degenerations. It is phrased in the context of matching fields and Grassmannians. 

\begin{Theorem}\cite[Theorem  11.4]{sturmfels1996grobner}\label{Theorem:Stu}
The set of Pl\"ucker forms  $\{P_I\}_{I\in \mathbf{I}_{k,n}}\subset \mathbb{K}[x_{ij}]$ 
is a Khovanskii basis with respect to a weight matrix $M$ if and only if $\ini_{w_M}(\I_{k, n}) = \I_\Lambda$, where $w_M$ is the weight vector on the variables $P_I$ induced by $M$ and ${\Lambda}$ is the matching field induced by $M$. 
\end{Theorem}

\section{Coherent matching fields and tropical hyperplane arrangements}\label{sec:arr}

In  \cite{fink2015stiefel}, Fink and Rinc\'on provide a link between  tropical hyperplane arrangements and coherent matching fields (and more generally multi-matching fields). We will summarize the facts needed here and refer the reader to \cite{fink2015stiefel} for more details.  In \cite{fink2015stiefel},  tropical hyperplane arrangements are described in tropical projective space. We do not require this level of generality here, therefore we simplify our considerations in the following summary. 

Let  $M=(a_{ij}) \in \mathbb{R}^{k \times n}$ be a weight matrix. For each $1 \leq j \leq n$ consider the  piecewise linear function $F_j: \mathbb{R}^{k-1} \to \mathbb{R}$, given by 
 \begin{eqnarray}\label{eqn:linearform}
 F_j (x) = \max \{ a_{1j}, a_{2j} + x_2, \dots, a_{kj}+ x_k\}. 
 \end{eqnarray}

In $\mathbb{R}^{k-1}$ let $\Sigma$ be the $k-2$ dimensional polyhedral fan whose top dimensional cones
are spanned by subsets of size $k-2$ of the vectors  $v_1, \dots, v_{k}$, where $v_1 = (1, \dots, 1)$ and 
$v_i = - e_{i-1}$ otherwise. For $k=3$ this amounts to three rays in the directions $(1, 1), (-1, 0)$ and $(0, -1)$ emanating from the origin. 
If $a_{1j} = 0$, then the  non-differentiability locus of $F_j$ is the fan $\Sigma \subset \mathbb{R}^{k-1}$ translated by the vector  $(-a_{2j} ,  \dots, -a_{kj}) \in \mathbb{R}^{k-1}$. Any  coherent matching field is  induced by a weight matrix $M$ whose first row consists of zeros. 
So we may assume that $a_{1j} = 0$ for all $j$.

 \begin{Remark}\label{rem:maxmin}
 In this paper we purposely use the minimum conventions for tropical arithmetic and the maximum conventions for tropical geometry. We do this to avoid the appearance of many minus signs when passing from the algebra of weight matrices to the geometry of tropical hyperplane arrangements. 
\end{Remark}

\begin{figure}
 \begin{center}
\begin{tikzpicture} [scale = .4, very thick = 1mm, every node/.style={inner sep=0,outer sep=0}]

\node (i) at (2,3)  [Cblack3] {};
\node (i0) at (7,8)  [Cwhite]{} ;
\node (i1) at (0,3)  [Cwhite] {\huge{$2$}};
\node (i2) at (2,0    ) [Cwhite]  {};

 \foreach \from/\to in {i/i0,i/i1,i/i2}
  \draw[black] (\from) -- (\to);

\node (j) at (4,2)  [Cblack3]  {};
\node (j0) at (9,7)  [Cwhite]{} ;
\node (j1) at (0,2)  [Cwhite] {\huge{$1$}};
\node (j2) at (4,0    ) [Cwhite]  {};

 \foreach \from/\to in {j/j0,j/j1,j/j2}
  \draw[black] (\from) -- (\to);

\node (j) at (8,4)  [Cblack3]  {};
\node (j0) at (10,6)  [Cwhite]{} ;
\node (j1) at (0,4)  [Cwhite] {\huge{$3$}};
\node (j2) at (8,0    ) [Cwhite]  {};

 \foreach \from/\to in {j/j0,j/j1,j/j2}
  \draw[black] (\from) -- (\to);

\node (j) at (18+ 4,6)  [Cblack3]  {};
\node (j0) at (18+6,8)  [Cwhite]{} ;
\node (j1) at (18+0,6)  [Cwhite] {\huge{$2$}};
\node (j2) at (18+4,0    ) [Cwhite]  {};

 \foreach \from/\to in {j/j0,j/j1,j/j2}
  \draw[black] (\from) -- (\to);

\node (j) at (18+ 6,4)  [Cblack3]  {};
\node (j0) at (18+8, 6)  [Cwhite]{} ;
\node (j1) at (18+0,4)  [Cwhite] {\huge{$1$}};
\node (j2) at (18+6,0    ) [Cwhite]  {};

 \foreach \from/\to in {j/j0,j/j1,j/j2}
  \draw[black] (\from) -- (\to);

\node (j) at (18+ 2,2)  [Cblack3]  {};
\node (j0) at (18+7,7)  [Cwhite]{} ;
\node (j1) at (18+0,2)  [Cwhite] {\huge{$3$}};
\node (j2) at (18+2,0    ) [Cwhite]  {};

 \foreach \from/\to in {j/j0,j/j1,j/j2}
  \draw[black] (\from) -- (\to);

\end{tikzpicture}\caption{\label{concurrent} Two arrangements of three tropical lines in $\mathbb{R}^2$. On the left the configuration intersects properly on the  right the three lines are concurrent.
}
\end{center}
\end{figure}
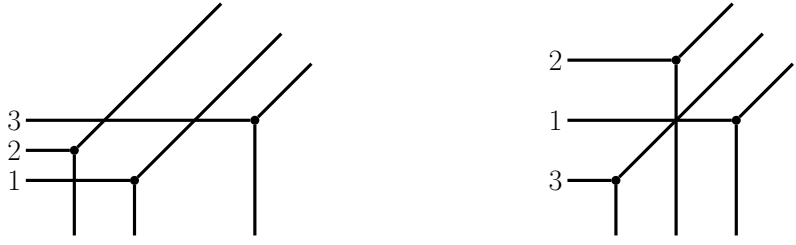

A $k \times n$  weight matrix $M \in \mathbb{R}^{k \times n}$,  whose first row consists of zeros,  produces an arrangement of tropical hyperplanes $\mathcal{A} = \{ H_1, \dots , H_n\}$ defined by the functions  $F_1, \dots , F_n$, whose coefficients come from $M$. 
To each $k-1$ dimensional cell $\tau$ of the  complement of the arrangement $\mathcal{A}$ in  $\mathbb{R}^{k-1}$ there is an associated \emph{covector} $c_{\tau} \in \mathcal{P}[n]^k$, where $\mathcal{P}[n]$ denotes the power set of $n$. 
The $i$-th entry of the covector $c_{\tau}$ is a subset $S_i \subset [n]$ corresponding to the collection of  hyperplanes in $\mathcal{A}$ which intersect the ray  $x + tv_i$ for $x \in \tau^{\circ}$ and $t \geq 0$. Here the vectors $v_i = -e_{i-1}$ for $i = 2, \dots, k-1$ and $v_1 = (1, \dots, 1)$. The \emph{coarse covector}
of a cell is simply the vector which records the sizes of the subsets of the covector.

\begin{Definition}
A collection of $k$ hyperplanes $H_1, \dots, H_k$  in $\mathbb{R}^{k-1}$ is said to   \emph{intersect properly} if $\cap_{i = 1}^k H_i = \emptyset$.
Equivalently, a collection of $k$ hyperplanes $H_1, \dots, H_k$  in $\mathbb{R}^{k-1}$ intersects properly if and only if there is a $k-1$ dimensional cell in the complement of $\cup_{i = 1}^k  H_i$  
whose coarse covector is  $(1, \dots, 1)$. 
\end{Definition}

\begin{Example}
Consider the two  $3 \times 3$ weight matrices, \[
M_1=\begin{bmatrix}
    0  & 0  & 0  \\
      4& 2 & 8 \\
        2& 3 & 4  \\
\end{bmatrix},
\qquad 
M_2=\begin{bmatrix}
    0  & 0  & 0  \\
     6 & 4 & 2  \\
      4  &  6 & 2   \\
\end{bmatrix}.
\]
The matrix $M_1$ corresponds to tropical lines intersecting properly in Figure~\ref{concurrent} (left) and the matrix $M_2$ corresponds to concurrent tropical lines  in Figure~\ref{concurrent} (right). 
\end{Example}

The following proposition can be extracted from \cite[Proposition 2.4]{DevSturm} and \cite[Propositions 5.11 and 5.12]{fink2015stiefel}. 
\begin{Proposition}\label{prop:initialPlucker}
Let $M$ be a $k \times n$ weight matrix such that for any size $k$ subset $I \subset [n]$,   the collection of hyperplanes $\{H_i\}_{i\in I}$ intersects properly. 
Then
$$\ini_{M} (P_{I}) =\text{sgn}({\Lambda(I)}) x_{1c_1}x_{2c_2}\dots x_{kc_{k}}$$
where $(c_1, \dots, c_{k})$ is the covector of the unique cell with coarse covector $(1, 1, \dots, 1)$ and 
$\text{sgn}({\Lambda(I)})$ is the sign of the permutation $i \mapsto \sigma(i)$, where $c_{\sigma(1)}<\cdots<c_{\sigma({k})}$.
\end{Proposition}

\begin{Corollary}
A $k \times n$ weight matrix  provides a monomial degeneration of the Pl\"ucker forms if and only if for any size $k$ subset $I \subset [n]$ the  collection of tropical hyperplanes $\{H_i\}_{i \in I}$ intersects properly. 
\end{Corollary}

Next we compare the matching field ideal $\I_{\Lambda}$ 
with the initial degeneration of the Pl\"ucker ideal $\I_{k, n}$ {with respect to the weights on the Pl\"ucker forms induced by $\Lambda$. }

\begin{Definition}\label{def:hex}
A coherent matching field $\Lambda \colon \mathbf{I}_{3, 6} \to S_3$ \emph{is hexagonal} if it is the matching field of a  tropical hyperplane arrangement whose unique cell with coarse covector $(2, 2, 2)$ is a hexagon. 
A matching field $\Lambda \colon \mathbf{I}_{3, n} \to S_3$ is  \emph{non-hexagonal} if for every size $6$ subset $J$ the matching field $\Lambda|_J$ is not hexagonal. 
\end{Definition}

For a homogeneous ideal $\I$ we let $\I_d$ denote the elements of degree $d$.

\begin{Proposition}\label{prop:deg2hex}
Let $M \in \mathbb{R}^{3 \times n}$ be a weight matrix such that $\ini_M(P_I)$ is a monomial for all Pl\"ucker forms $P_I$. Then 
$(\I_{\Lambda})_2 = (\ini_{w_M}(\I_{3,n}))_2$ if and only if 
the matching field is non-hexagonal. 
\end{Proposition}

Before giving the proof we pause to illustrate the condition presented in the above proposition with two examples.

\begin{figure}
 \begin{center}
\begin{tikzpicture} [scale = .3, very thick = 1mm, every node/.style={inner sep=0,outer sep=0}]

\node (w1) at (-.3,1) [Cwhite]  {$6$};
\node (w1) at (-.3,7) [Cwhite]  {$3$};
\node (w1) at (-.3,3) [Cwhite]  {$5$};
\node (w1) at (-.3,11) [Cwhite]  {$1$};
\node (w1) at (-.3,9) [Cwhite]  {$2$};
\node (w1) at (-.3,5) [Cwhite]  {$4$};

\node (k) at (1,1)  [Cblack3] {};
\node (k0) at (8,8)  [Cwhite]{} ;
\node (k1) at (0,1)  [Cwhite] {};
\node (k2) at (1,0) [Cwhite]  {};

 \foreach \from/\to in {k/k0,k/k1,k/k2}
  \draw[black] (\from) -- (\to);
  
\node (i) at (2,3)  [Cblack3] {};
\node (i0) at (8,9)  [Cwhite]{} ;
\node (i1) at (0,3)  [Cwhite] {};
\node (i2) at (2,0    ) [Cwhite]  {};

 \foreach \from/\to in {i/i0,i/i1,i/i2}
  \draw[black] (\from) -- (\to);

\node (j) at (3,5)  [Cblack3]  {};
\node (j0) at (8,10)  [Cwhite]{} ;
\node (j1) at (0,5)  [Cwhite] {};
\node (j2) at (3,0    ) [Cwhite]  {};

 \foreach \from/\to in {j/j0,j/j1,j/j2}
  \draw[black] (\from) -- (\to);

\node (l) at (4,7) [Cblack3]  {};
\node (l0) at (8,11)  [Cwhite]{} ;
\node (l1) at (0,7)  [Cwhite] {};
\node (l2) at (4,0    ) [Cwhite]  {};

 \foreach \from/\to in {l/l0,l/l1,l/l2}
  \draw[black] (\from) -- (\to);

\node (l) at (5,9)  [Cblack3]  {};
\node (l0) at (8,12)  [Cwhite]{} ;
\node (l1) at (0,9)  [Cwhite] {};
\node (l2) at (5,0    ) [Cwhite]  {};

 \foreach \from/\to in {l/l0,l/l1,l/l2}
  \draw[black] (\from) -- (\to);
  
  \node (l) at (6,11)  [Cblack3]  {};
\node (l0) at (8,13)  [Cwhite]{} ;
\node (l1) at (0,11)  [Cwhite] {};
\node (l2) at (6,0    ) [Cwhite]  {};

 \foreach \from/\to in {l/l0,l/l1,l/l2}
  \draw[black] (\from) -- (\to);

  \end{tikzpicture}
 \hspace{2cm}
\begin{tikzpicture} [scale = .3, very thick = 1mm, every node/.style={inner sep=0,outer sep=0}]

\node (w1) at (-.3,1) [Cwhite]  {$6$};
\node (w1) at (-.3,7) [Cwhite]  {$3$};
\node (w1) at (-.3,3) [Cwhite]  {$5$};
\node (w1) at (-.3,11) [Cwhite]  {$1$};
\node (w1) at (-.3,9) [Cwhite]  {$2$};
\node (w1) at (-.3,5) [Cwhite]  {$4$};

\node (k) at (2,11)  [Cblack3] {};
\node (k0) at (4,13)  [Cwhite]{} ;
\node (k1) at (0,11)  [Cwhite] {};
\node (k2) at (2,0) [Cwhite]  {};

 \foreach \from/\to in {k/k0,k/k1,k/k2}
  \draw[black] (\from) -- (\to);
  
\node (i) at (1,9)  [Cblack3] {};
\node (i0) at (4,12)  [Cwhite]{} ;
\node (i1) at (0,9)  [Cwhite] {};
\node (i2) at (1,0    ) [Cwhite]  {};

 \foreach \from/\to in {i/i0,i/i1,i/i2}
  \draw[black] (\from) -- (\to);

\node (j) at (3,1)  [Cblack3]  {};
\node (j0) at (8,6)  [Cwhite]{} ;
\node (j1) at (0,1)  [Cwhite] {};
\node (j2) at (3,0    ) [Cwhite]  {};

 \foreach \from/\to in {j/j0,j/j1,j/j2}
  \draw[black] (\from) -- (\to);

\node (l) at (4,3) [Cblack3]  {};
\node (l0) at (8,7)  [Cwhite]{} ;
\node (l1) at (0,3)  [Cwhite] {};
\node (l2) at (4,0    ) [Cwhite]  {};

 \foreach \from/\to in {l/l0,l/l1,l/l2}
  \draw[black] (\from) -- (\to);

\node (l) at (5,5)  [Cblack3]  {};
\node (l0) at (8,8)  [Cwhite]{} ;
\node (l1) at (0,5)  [Cwhite] {};
\node (l2) at (5,0    ) [Cwhite]  {};

 \foreach \from/\to in {l/l0,l/l1,l/l2}
  \draw[black] (\from) -- (\to);
  
  \node (l) at (6,7)  [Cblack3]  {};
\node (l0) at (8,9)  [Cwhite]{} ;
\node (l1) at (0,7)  [Cwhite] {};
\node (l2) at (6,0    ) [Cwhite]  {};

 \foreach \from/\to in {l/l0,l/l1,l/l2}
  \draw[black] (\from) -- (\to);

\node (w1) at (16.3,11) [Cwhite]  {$1$};
\node (w1) at (16.3,5) [Cwhite]  {$4$};
\node (w1) at (16.3,1) [Cwhite]  {$6$};
\node (w1) at (16.3,9) [Cwhite]  {$2$};
\node (w1) at (16.3,7) [Cwhite]  {$3$};
\node (w1) at (16.3,3) [Cwhite]  {$5$};

\node (k) at (18.6,9)  [Cblack3] {};
\node (k0) at (22.6,12.86)  [Cwhite]{} ;
\node (k1) at (16.6,9)  [Cwhite] {};
\node (k2) at (18.6,0) [Cwhite]  {};

 \foreach \from/\to in {k/k0,k/k1,k/k2}
  \draw[black] (\from) -- (\to);
  
\node (i) at (17.6,7)  [Cblack3] {};
\node (i0) at (22.6,11.49)  [Cwhite]{} ;
\node (i1) at (16.6,7)  [Cwhite] {};
\node (i2) at (17.6,0    ) [Cwhite]  {};

 \foreach \from/\to in {i/i0,i/i1,i/i2}
  \draw[black] (\from) -- (\to);

\node (j) at (19.6,11)  [Cblack3]  {};
\node (j0) at (22.6,14)  [Cwhite]{} ;
\node (j1) at (16.6,11)  [Cwhite] {};
\node (j2) at (19.6,0  ) [Cwhite]  {};

 \foreach \from/\to in {j/j0,j/j1,j/j2}
  \draw[black] (\from) -- (\to);

\node (l) at (20.6,1) [Cblack3]  {};
\node (l0) at (24.6,5)  [Cwhite]{} ;
\node (l1) at (16.6,1)  [Cwhite] {};
\node (l2) at (20.6,0  ) [Cwhite]  {};

 \foreach \from/\to in {l/l0,l/l1,l/l2}
  \draw[black] (\from) -- (\to);

\node (l) at (21.6,3)  [Cblack3]  {};
\node (l0) at (24.6,6)  [Cwhite]{} ;
\node (l1) at (16.6,3)  [Cwhite] {};
\node (l2) at (21.6,0  ) [Cwhite]  {};

 \foreach \from/\to in {l/l0,l/l1,l/l2}
  \draw[black] (\from) -- (\to);
  
  \node (l) at (22.6,5)  [Cblack3]  {};
\node (l0) at (24.6,7)  [Cwhite]{} ;
\node (l1) at (16.6,5)  [Cwhite] {};
\node (l2) at (22.6,0   ) [Cwhite]  {};

 \foreach \from/\to in {l/l0,l/l1,l/l2}
  \draw[black] (\from) -- (\to);

  \end{tikzpicture}

 \end{center}\caption{On the left a tropical hyperplane arrangement yielding the diagonal matching field of size $3 \times 6$, in the middle an arrangement yielding the block diagonal matching field $\BLambda_{2,4}$ and on the right a block diagonal matching field $\BLambda_{3,3}$. See Section \ref{sec:block} for the definition of a block diagonal matching field.} \label{fig:diag} 
   \end{figure}
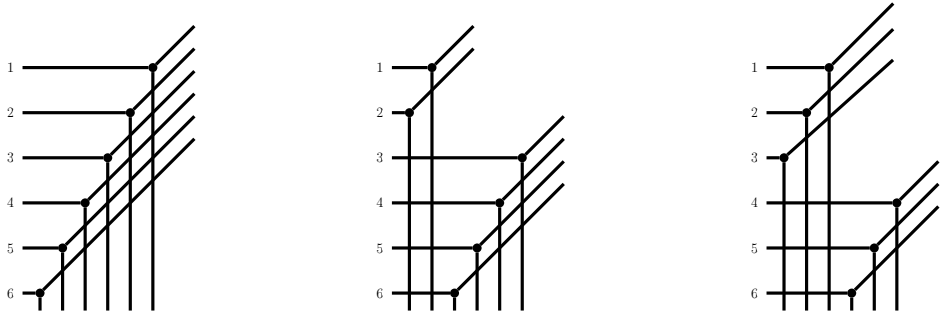

\begin{Example}[Diagonal matching field]
\label{exam:diag}
The associated hyperplane arrangement of the following weight matrix is depicted in Figure~\ref{fig:diag} (left).
\[
M=\begin{bmatrix}
    0  & 0  & 0  & 0  & 0  & 0  \\
      6& 5  & 4 & 3  & 2  & 1\\
       11 & 9 & 7 & 5 & 3  & 1 \\
\end{bmatrix}
\]

The initial terms are $x_{1i}x_{2j}x_{3k}$ for $1\leq i<j<k\leq 6$.  The following $\Lambda$-tableaux indicate the row-wise equal tableaux which give all the binomial relations in $\I_{\Lambda}$ of multi-degree $e_J$ for $|J|=6$, 
$$
\begin{array}{c}1  \\ 3 \\5\end{array} 
\begin{array}{c}2 \\ 4 \\6\end{array}= 
\begin{array}{c}1  \\3 \\6\end{array}
\begin{array}{c}2 \\4 \\5\end{array} = 
\begin{array}{c}1  \\4 \\5\end{array} 
\begin{array}{c}2 \\3 \\6\end{array} = 
\begin{array}{c}1  \\4 \\6\end{array} 
\begin{array}{c} 2\\3 \\5\end{array} , \quad 
\begin{array}{c}1 \\2 \\5 \end{array} 
\begin{array}{c}3 \\4 \\6\end{array}= 
\begin{array}{c}1 \\2 \\6\end{array} 
\begin{array}{c}3 \\4 \\5\end{array},\quad 
\begin{array}{c}1  \\3 \\4\end{array} 
\begin{array}{c}2 \\5 \\6\end{array} = 
\begin{array}{c}1  \\5 \\6\end{array} 
\begin{array}{c}2 \\3 \\4\end{array}.
$$

Notice that in each of the equivalence classes of the quadratic monomials listed above the first monomial listed is a semi-standard tableaux, i.e., all rows are in weakly increasing order and the columns are strictly increasing and it is the only semi-standard tableaux of that equivalence class. 
The other quadratic terms of multi-degree $(1, \dots, 1)$ are
$$
 \begin{array}{c}1 \\2 \\3\end{array} 
\begin{array}{c}4 \\5 \\6\end{array} 
\qquad \text{and} \qquad
\begin{array}{c}1 \\2 \\4\end{array}
\begin{array}{c}3 \\5 \\6\end{array}.
$$
Notice that they are independent in $\mathbb{K}[P_I]$. 

\end{Example}

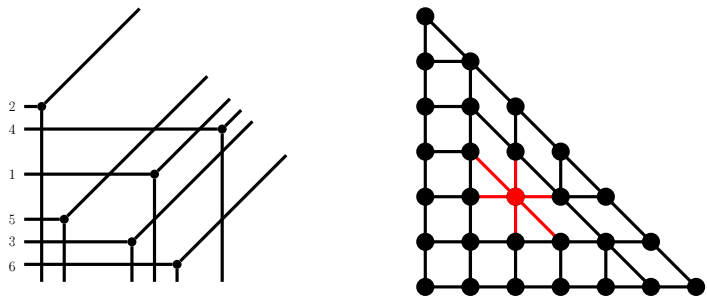
\begin{figure}
 \begin{center}
\begin{tikzpicture} [scale = .3, very thick = 1mm, every node/.style={inner sep=0,outer sep=0}]

\node (w1) at (-.3,.93) [Cwhite]  {$6$};
\node (w1) at (-.3,2) [Cwhite]  {$3$};
\node (w1) at (-.3,3) [Cwhite]  {$5$};
\node (w1) at (-.3,5) [Cwhite]  {$1$};
\node (w1) at (-.3,7) [Cwhite]  {$4$};
\node (w1) at (-.3,8) [Cwhite]  {$2$};

\node (k) at (1,8)  [Cblack3] {};
\node (k0) at (5.5,12.5)  [Cwhite]{} ;
\node (k1) at (0,8)  [Cwhite] {};
\node (k2) at (1,0) [Cwhite]  {};

 \foreach \from/\to in {k/k0,k/k1,k/k2}
  \draw[black] (\from) -- (\to);
  
\node (i) at (2,3)  [Cblack3] {};
\node (i0) at (8.5,9.5)  [Cwhite]{} ;
\node (i1) at (0,3)  [Cwhite] {};
\node (i2) at (2,0    ) [Cwhite]  {};

 \foreach \from/\to in {i/i0,i/i1,i/i2}
  \draw[black] (\from) -- (\to);

\node (j) at (5,2)  [Cblack3]  {};
\node (j0) at (10.5,7.5)  [Cwhite]{} ;
\node (j1) at (0,2)  [Cwhite] {};
\node (j2) at (5,0    ) [Cwhite]  {};

 \foreach \from/\to in {j/j0,j/j1,j/j2}
  \draw[black] (\from) -- (\to);

\node (l) at (9,7) [Cblack3]  {};
\node (l0) at (10,8)  [Cwhite]{} ;
\node (l1) at (0,7)  [Cwhite] {};
\node (l2) at (9,0    ) [Cwhite]  {};

 \foreach \from/\to in {l/l0,l/l1,l/l2}
  \draw[black] (\from) -- (\to);

\node (l) at (6,5)  [Cblack3]  {};
\node (l0) at (9.5,8.5)  [Cwhite]{} ;
\node (l1) at (0,5)  [Cwhite] {};
\node (l2) at (6,0    ) [Cwhite]  {};

 \foreach \from/\to in {l/l0,l/l1,l/l2}
  \draw[black] (\from) -- (\to);

\node (l) at (7,1) [Cblack3]  {};
\node (l0) at (12,6)  [Cwhite]{} ;
\node (l1) at (0,1)  [Cwhite] {};
\node (l2) at (7,0    ) [Cwhite]  {};

 \foreach \from/\to in {l/l0,l/l1,l/l2}
  \draw[black] (\from) -- (\to);

\draw[red](4 + 18,2) -- (4+18,6);    
\draw[red] (2+ 18,4) -- (6+18,4); 
\draw[red] (6 + 18,2) -- (2+18,6); 

 \foreach \x in {0, 1,..., 6}{
  	\foreach \y in {0, ...,\x}{
 		\node[draw,circle,inner sep=2pt,fill, black] at (12 - 2*\x + 18, 2*(\y ) {};
	}
}	
    
\draw[black] (0 + 18,0) -- (12 + 18,0);    
\draw[black] (0 + 18,0) -- (18,12);            
\draw[black] (12 + 18,0) -- (18,12);            

\draw[black] (2+ 18,0) -- (20,10);     
\draw[black] (2+ 18,2) -- (28,2);     

\draw[black] (2 + 18,8) -- (10 + 18,0);    

\draw[black] (8+ 18,4) -- (6+18,4); 
\draw[black] (6+ 18,6) -- (6+18,4);
 
\draw[black](4 + 18,8) -- (4+18,6);  

 \foreach \x in {1,2,..., 4}{
  \draw[black] (2*\x + 18, 0) -- (2*\x + 18, 2) ;}
  
   \foreach \y in {1,2,..., 5}{
  \draw[black] (18, 2*\y) -- (2 + 18, 2*\y) ;}

\node[draw,circle,inner sep=2pt,fill, red] at (2*2 + 18, 2*2) {};

\end{tikzpicture}\caption{\label{fig:badhex} The tropical hyperplane arrangement from the weight matrix in Example \ref{exam:hexagon} on the left and its dual regular subdivision on the $2$ dimensional simplex of size $6$.}
\end{center}
\end{figure}

\begin{Example}[Hexagonal matching field]\label{exam:hexagon}
The associated hyperplane arrangement of the following weight matrix

\[
M=\begin{bmatrix}
    0  & 0 & 0  & 0  & 0 & 0  \\
     6 & 1  & 5  & 9 & 2 &  7 \\
       5  & 8 & 2  & 7  & 3  & 1 \\
\end{bmatrix}
\]
is depicted in Figure~\ref{fig:badhex} (left). The initial terms of the Pl\"ucker forms are:
\[
  123, 421,   125,  126,  413,  153,  136,   451, 416, 156,  423,  523, 326, 425, 426, \]
  \[
  526,   453,   436,   356,  \text{ and }  456,
   \]
where by $ijk$ we mean $x_{1i}x_{2j}x_{3k}$.  
The following $\Lambda$-tableaux indicate the row-wise equal tableaux which give all the binomial relations in $\I_{\Lambda}$ of the form $e_J$ for $|J|=6$. 
$$\begin{array}{c}1 \\2 \\3\end{array} 
\begin{array}{c}4 \\5 \\6\end{array} = 
\begin{array}{c}4 \\2 \\3\end{array} 
\begin{array}{c}1 \\5 \\6\end{array}=
\begin{array}{c}1 \\5 \\3\end{array} 
\begin{array}{c}4 \\2 \\6\end{array}= 
\begin{array}{c}1 \\2 \\6\end{array} 
\begin{array}{c}4 \\5 \\3\end{array},\quad 
\begin{array}{c}5  \\2 \\3\end{array} 
\begin{array}{c}4 \\1 \\6\end{array} = 
\begin{array}{c}5  \\2 \\6\end{array} 
\begin{array}{c}4 \\1 \\3\end{array},
$$
$$
\begin{array}{c}4  \\2 \\5\end{array} 
\begin{array}{c}1 \\3 \\6\end{array} = 
\begin{array}{c}4  \\3 \\6\end{array} 
\begin{array}{c}1 \\2 \\5\end{array},\quad 
\begin{array}{c}4  \\2 \\1\end{array} 
\begin{array}{c}3 \\5 \\6\end{array} = 
\begin{array}{c}4  \\5 \\1\end{array} 
\begin{array}{c}3 \\2 \\6\end{array}.$$
\end{Example}
Notice that compared with Example \ref{exam:diag}, there is an additional binomial relation listed.

\begin{Example}
The matching field ideal in Example~\ref{exam:diag}, is generated by 35 binomials and it is equal to the initial ideal of $\I_{3,n}$.
However, in Example~\ref{exam:hexagon}, the ideal $\I_\Lambda$ is generated by 36 binomials. More precisely, the relation $P_{523}P_{416}-P_{526}P_{413}$ is in $\I_\Lambda$, but not in the initial ideal $\ini_{w_M}(\I_{3,n})$.
\end{Example}

\begin{Lemma}\label{lem:222}
Let $M$ be a $3 \times n$ weight matrix providing a monomial degeneration of the Pl\"ucker forms. 
Then for any size $6$ subset $J \subset [n]$, the $(2, 2, 2)$-cell of $\mathcal{A}|_J$ determines
 the initial terms of exactly $8$ Pl\"ucker forms. 
 Moreover, these $8$ initial terms come in $4$ pairs  which produce 
 quadratic relations in $\I_{\Lambda}$.

\end{Lemma}

\begin{proof}
Suppose that the  covector of the $(2, 2, 2)$-cell  is $(S_1,S_2, S_3)$
where $|S_i| = 2$ for all $i$.  Then choosing  $s_i \in S_i$ for $i = 1, 2, 3$ we obtain $\ini_M(P_{s_1s_2s_3}) =  \pm x_{1s_1}x_{2s_2}x_{3s_3}$ 
by Proposition \ref{prop:initialPlucker}.  Therefore the first claim follows. 

For simplicity we can assume that $J = \{1, \dots, 6\}$ and that $$(S_1, S_2, S_3) = (\{1, 4\}, \{2, 5\}, \{3, 6\}),$$ as in the case for the hexagon cell in Figure~\ref{fig:badhex}. Then the following  $\Lambda$-tableaux are all row-wise equal and give rise to 3 binomial relations in $\I_{\Lambda}$, 
$$\begin{array}{c}1 \\2 \\3\end{array} 
\begin{array}{c}4 \\5 \\6\end{array} = 
\begin{array}{c}4 \\2 \\3\end{array} 
\begin{array}{c}1 \\5 \\6\end{array}=
\begin{array}{c}1 \\5 \\3\end{array} 
\begin{array}{c}4 \\2 \\6\end{array}= 
\begin{array}{c}1 \\2 \\6\end{array} 
\begin{array}{c}4 \\5 \\3\end{array}.
$$
This completes the proof.
\end{proof}

\begin{proof}[Proof of Proposition~\ref{prop:deg2hex}.] 
Suppose we have an arrangement of $6$ tropical lines in $\mathbb{R}^2$.
The cell  which has coarse covector  $(2, 2, 2)$ is a hexagon if and only if the edges with endpoints  $(2, 2, 2) - (\sigma(1), \sigma(2), \sigma(3))$ for all $\sigma \in S_3$  are present in the 
dual subdivision of $6\Delta_2$. See the right hand side of Figure \ref{fig:badhex}.

Suppose without loss of generality that the covector of the hexagon cell is $(\{1,4\},\{2, 5\},$ $\{3, 6\})$, as it is for example in Figure~\ref{fig:badhex}.
If the cell dual to $(2, 2, 2)$ is a hexagon, then up to the appropriate labeling there are covectors, 
$$(\{1,4\}, \{2\}, \{3, 5, 6\}) \qquad\text{and}\qquad (\{1,4\}, \{ 2, 3, 5\}, \{6\}). $$
From this pair of covectors we obtain the quadratic relation $$\begin{array}{c}1 \\2 \\5\end{array} 
\begin{array}{c}4 \\3 \\6\end{array} = 
\begin{array}{c}4 \\2 \\5\end{array} 
\begin{array}{c}1 \\3 \\6\end{array}
$$
in the ideal $\I_{\Lambda}$. 

There are 4 other lattice points that are endpoints of the 6 segments of the subdivision dual to the hexagon. 
They come in 2 pairs formed by the points which are on the same line. For each of these pairs we obtain a new 
 independent quadratic relation in $\I_{\Lambda}$ as above. 
 
Taking into account the quadratic relations from  Lemma \ref{lem:222}  as well, we can conclude that the dimension of the $e_1 + \dots +e_6$ graded piece of the coordinate ring of the toric variety $\I_{\Lambda}$ is at most $4$. However, the dimension of this piece for the coordinate ring of  $\I_{3, n}$  and 
hence also of $\ini_{w_M}(\I_{3, n})$  is $5$. This dimension is given by the number of semi-standard tableaux with content $\{1, 2, 3, 4, 5, 6\}$ of size $3 \times 2$. 

For the other direction, we again consider the multi-grading on the coordinate ring of $I_{\Lambda}$. The degree $2$ part of this coordinate ring has elements which have two types with respect to the multi-grading. They are either $e_J$ or  $e_J + e_i - e_j$ for $i, j \in J$ and $i \neq j $ where $J$ is a size $6$ subset. The dimension of the $e_J + e_i - e_j$ piece of the coordinate ring   is of the correct dimension for any $J$ and $i, j$. This follows from \cite[Corollary 4.4]{SpeyerSturmfels}.

{Finally, we consider  the terms with  $e_J$ multi-grading for a  $J$ with $|J| = 6$. If $\Lambda$ is not hexagonal, then the dual subdivision of a tropical hyperplane arrangement inducing $\Lambda$ must be missing one of the possible edges with endpoint $(2, 2) \in 6 \Delta$. Proceeding case by case, we can verify the statement of the proposition.  }
\end{proof}

\begin{proof}[Proof of Theorem \ref{Theorem:iff}]

Assume that  $\I_\Lambda $ is quadratically generated. If the matching field $\Lambda$ is hexagonal, then by Proposition \ref{prop:deg2hex}  there is a size $6$ subset $J \subset [n]$ such that $(\I_{\Lambda})_2 \neq (\ini_{w_M}(\I_{3,n}))_2$. 
However, by \cite[Theorem  11.4]{sturmfels1996grobner}, equality of the  ideals $\I_{\Lambda}$ and $\ini_{w_M}(\I_{3,n})$  is a necessary and sufficient condition for the Pl\"ucker forms to be a Khovanskii basis. This proves one direction.

For the other direction, 
we compare the ideals $\ini_{w_M}(\I)$ and $\I_\Lambda$, and then we complete the proof by applying Theorem \ref{Theorem:Stu}.
Since $\I_\Lambda$ is quadratically generated, there are  two types  of generators determined by their multi-degrees. There are generators of type $e_J$ or of type $e_J + e_i - e_j$ where $|J| = 6$ and $i, j \in J$. We show that the generators of $\I_{\Lambda}$ are included in $(\ini_{w_M}(\I_{3,n}))_2$ by considering each type. 

Firstly, the generators of type $e_J +e_i -e_j$ can be reduced to the case of $\Gr(2,5)$.
 In this case the statement holds since the Pl\"ucker forms are a Khovanskii basis with respect to any  coherent matching field $\Lambda' \colon \mathbf{I}_{2, 5} \to S_2$ \cite{SpeyerSturmfels}.

For generators of type $e_J$ where $|J| = 6$, we reduce to the situation of $\Gr(3, 6)$ and matching fields of the form $\Lambda' \colon \mathbf{I}_{3, 6} \to S_3$. 
Combining Proposition \ref{prop:deg2hex}  and 
 \cite[Theorem  11.4]{sturmfels1996grobner} shows that the Pl\"ucker forms are a Khovanskii basis with respect to any of these coherent matching fields. Therefore the subduction algorithm terminates with a constant for any generator of type $e_J$ when the restriction to $J$ is not hexagonal. By again applying   \cite[Corollary 11.5]{sturmfels1996grobner} we prove the other direction and our theorem. 
\end{proof}

Following Theorem \ref{Theorem:iff}, we are interested in determining when a $3\times n$ matching field ideal is quadratically generated.

\begin{Example}
The ideal of the diagonal matching field from Example \ref{ex:diagonal1} is quadratically generated,
see \cite[Theorem~14.16]{CCA}. 
\end{Example}

\begin{Example}\label{exa:non-coherent}
Consider the $2\times 6$ matching field $\Lambda$ that assigns the transposition $(1 2)$ for sets $I \in \{\{1, 4\}, \{2,3\}, \{3,6\}, \{4,5\}\}$ and the identity permutation otherwise. 
A minimal generator of the matching field ideal is 
$$\begin{array}{c}1 \\2\end{array} 
\begin{array}{c}3 \\4 \end{array}
\begin{array}{c}5 \\6 \end{array} 
= 
\begin{array}{c}1 \\6 \end{array} 
\begin{array}{c}3 \\2\end{array} 
\begin{array}{c}5 \\4 \end{array}.
$$
Therefore, this matching field ideal is not quadratically generated.

\end{Example}

\begin{Remark}The matching field of the non-quadratically generated ideal in Example~\ref{exa:non-coherent} is not coherent since it does not arise from a total ordering on the set $[n]$. Our smallest known examples of coherent matching fields whose ideals are not quadratically generated are of size $3 \times 8$ and were found via a random search. 
\end{Remark}

Before presenting the proof of Theorem \ref{Theorem:subhex} we introduce the notion of submatching field. 
Before  defining  hexagonal submatching fields,  recall the notion of a matching field  being pointed on a subset $S$ of $[n]$ from Example \ref{ex:pointed1}.

\begin{Definition}\label{def:submatching}
Given a matching field $\Lambda$ and two subsets $S \subset T \subset [n]$, the \emph{submatching field} $\Lambda|_{T-S,T}$ of $\Lambda$
is obtained by restricting $\Lambda$ to subsets $I$ of $[n]$ with $S\subset I\subset T$ and restricting the matching to $I \backslash S$.

The submatching field $\Lambda|_{T-S,T}$ is \emph{hexagonal} if it is a size $3 \times 6$ hexagonal matching field and  $\Lambda$ is pointed on $S \subset [n]$. 
\end{Definition}

\begin{proof}[Proof of Theorem  \ref{Theorem:subhex}]
{Let  $\Lambda|_{T-S,T}$  be a hexagonal  submatching field  of $\Lambda$. Consider the  graded piece of the Pl\"ucker algebra consisting of degree $2$ monomials in the variables $P_I$ such that $ S \subset I \subset T$. This vector space has the same dimension as the degree $2$ graded piece of $\mathcal{A}_{3, 6}$, and this is $5$ dimensional. However, the analogous graded piece of $\mathbb{K}[x_{ij}]/\mathcal{I}_{\Lambda}$ is only $4$ dimensional since  
 $\mathcal{I}_{\Lambda}$ consists of the list of  binomials from Example \ref{exam:hexagon}. 
The graded Hilbert functions of $\mathcal{A}_{k, n}$ and $\mathbb{K}[x_{ij}]/\mathcal{I}_{\Lambda}$ are not equal and therefore the matching field $\Lambda$ does not produce a toric degeneration of $\Gr(k,n)$.}   
\end{proof}

\section{Block diagonal matching fields}\label{sec:block}

In this section we describe a family of coherent matching fields of size $3 \times n$ whose toric ideals are generated in degree $2$ and therefore yield toric degenerations and Khovanskii bases of $\Gr(3,n)$.

Consider a sequence of positive numbers $ a_1, a_2, \dots,  a_{r}$ so that $\sum_{i = 1}^r a_i = n$. For $1 \leq s\leq r$ set 
$I_s = \{\alpha_{s-1}+1, \alpha_{s-1} +2, \dots, \alpha_{s}\}$,
where $\alpha_{s} = \sum_{i = 1}^{s} a_{i }$ and $\alpha_0 = 0$. 

\begin{Definition}\label{def:blockdiagcolumns}

The block diagonal matching field of size $3 \times n$  corresponding to a collection ${\bf a} = \{a_1, \dots, a_{r}\}$ satisfying  $\sum_{i = 1}^r a_i = n$ is denoted $\BLambda_{\bf{a}}$. This matching field is defined by:
\begin{enumerate}

\item $\BLambda_{\bf{a}}(I) = \id$ if $|I \cap I_s| \geq 2$ where $s$ is the minimal $t$ such that $I_t \cap I \neq \emptyset$;

\item \label{transpose} $\BLambda_{\bf{a}}(I) = (12)$ if $|I \cap I_s| = 1$ where $s$ is the minimal $t$ such that $I_t \cap I \neq \emptyset$. 

 \end{enumerate}
 A $2$-block diagonal matching field is a block diagonal matching field with $r = 2$. 
\end{Definition}

\begin{Example}
Consider the case when $a_1 = 1 $ and  $a_2 = n-1$. Then  $I_1 = \{1\}$ and $I_2 = \{2, \dots , n\}$.  
Then $\BLambda_{1, n-1}(I) = \text{id}$ if and only if $I \subset I_2$. Otherwise, we have $1 \in I$ and $1$ appears in the second row of the table.
The matching field  $\BLambda_{1,n-1}$ is isomorphic to a pointed matching field $\Lambda$. This isomorphism is given by acting on $[n]$ by the transposition $(12)$. 
In fact, the $\Lambda$-tableaux are then the PBW-tableaux from \cite{Feigen}. \end{Example}

\begin{Remark}\label{rem:higherblock}
Block diagonal matching fields can be generalised to size $k \times n$. The ideals of all $2$-block diagonal matching fields are also  quadratically generated. This can be proved in the same way as Theorem \ref{Theorem:2block}.  However, we  cannot prove analogues of Corollary \ref{cor:blockdiag} for Grassmannians $\Gr (k, n)$ for $k >3$ since in these cases  quadratic generation of the initial ideals does not directly imply that the initial degeneration is toric. 
\end{Remark}

 In general  there is a $\mathbb{Z}^4$  grading  given by the number of elements of type $I_1$ in different rows of a tableau. A $3\times d$ tableau $T$ is of degree $(\alpha, \beta,\gamma, d-\alpha-\beta-\gamma)$ 
 where  
\begin{itemize}
\item $\alpha=|\{\text{Content of row 3 of } T\}\cap I_1| = |\{I\in T: \ |I \cap I_1| = 3\}|$ 
\item $\beta=|\{\text{Content of row 1 of } T\}\cap I_1|-\alpha = |\{I\in T: \ |I \cap I_1| = 2\}|$ 
\item $\gamma=|\{\text{Content of row 2 of } T\}\cap I_1|-\alpha-\beta = |\{I\in T: \ |I \cap I_1| = 1\}|$ 
\end{itemize}
For two $\Lambda$-tableaux $T,T'$ which are row-wise equal, these numbers are equal. This implies the following lemma.
  
 \medskip
 
\begin{Lemma}\label{lem:homo}
The ideal of a block diagonal matching field
has a $\mathbb{Z}^4$ grading given by $(\alpha, \beta,\gamma,  d-\alpha-\beta-\gamma)$
from above. 
 \end{Lemma}

\begin{proof}[Proof of Theorem \ref{Theorem:2block}]

Consider a binomial relation obtained from two $\Lambda$-tableaux  $T$ and $T'$ of size $3 \times d$ where $d >2$  whose contents are row-wise equal. By applying quadratic changes to the tableaux (changes involving only two columns) we will reduce the degree of this relation thus proving that the matching field ideal is quadratically generated. 

Given a $\Lambda$-tableau $T$, arrange the columns so that the first columns are those for which the matching field assigns the identity permutation and to the last column the matching field assigns the transposition $(12)$. Let $C$ denote the subtableau formed by the first columns and let $D$ denote the subtableau formed by the last columns. 

The tableaux $C$ and $D$ can each be put into semi-standard format. 
In other words, we can rearrange both $C$ and $D$ so that all  rows are in weakly increasing order, and the columns of $C$ are strictly increasing, whereas the columns of $D$ are arranged so that the first and second entries are permuted from the diagonal order. 

Now given  a binomial relation obtained from two $\Lambda$-tableaux  $T$ and $T'$. We assume that $T$ (respectively $T'$) is organized as a  pair of subtableaux $C$, $D$ (respectively  $C'$, $D'$) satisfying the requirements described above. 
 If the first columns of $T$ and $T'$ are equal, then we can cancel them from the binomial relation and it is not a minimal generator.  
 We let $I$ and $I'$ denote the first columns of $T$ and $T'$, respectively. 
Otherwise by Lemma \ref{lem:homo} the matching field relations are homogeneous with respect to the $\mathbb{Z}^4$ grading and so $| I \cap I_1 | = |I' \cap I_1|$.

Case 1: Suppose that $| I \cap I_1 | = |I' \cap I_1| = 3$. In this case, the columns $I$ and $I'$ could only differ in the second row. Suppose the  entries of the second row of $I$ and $I'$ are $j $ and $j'$, respectively. We can also assume that  $j <j'$. 
Then there must be a $j$ in the second row of the tableaux $D'$ since the contents are row-wise equal and $C'$ is in weakly increasing order. Then swap the positions of $j$ and $j'$ in the second row of $T'$ so that the first columns of $T$ and $T'$ now agree. Notice that we can exchange the position of  $j$ with that of $j'$ since $j, j' \in I_1$ and $j'$ was originally in the second row  of a  column whose the first and third entries were in $I_2$.  

Case 2: Suppose that $| I \cap I_1 | = |I' \cap I_1| =2$.  In this case, the columns $I$ and $I'$  may only differ in the second and third rows but not in the first. Assume the column $I$  is $i, j, r$ and the column  $I'$ is  $i, j', r'$. If $j <j'$ then just as above there must be a $j$ in the second row of $D'$. We have that $j <r'$ since $j$ is in block $1$ and $r'$ is in block $2$, so   we can swap the positions of $j $ and $j'$ in the second row of $T'$. 

Assume now that $j = j'$, without loss of generality we can suppose that $r<r'$. Then there is an $r$ in the third row of $D'$. Suppose that the column containing $r$ is $s, t, r$. Then we can swap the positions  of $r$ and $r'$ since $t<r<r'$ and  we can place $r$ in the last row. Now the two first terms are equal and hence, the binomial is not a minimal generator.

Case 3: Suppose that $| I \cap I_1 | = |I' \cap I_1| = 1$.  In this case, the tableaux $C$ and $C'$ are empty. Then the first two columns must be equal since $T = D$ and $T'= D'$ and they are both in (transposed) semi-standard form.

Case 4: Suppose that $| I \cap I_1 | = |I' \cap I_1| =0$. In this case, the entries of $I$ and $I'$ can only differ in the first and third row. Suppose that the column $I$ is $r, s, t$ and that the column $I'$ is $r', s, t'$.  
Therefore $r, r'<s<t, t'$ and we can assume that $r' <r$. 
Then there is a column in $T'$ with $r$ in the first row and we can  swap $r$ and $r'$ in the first row of $T'$. 
Thus we may assume that $r = r'$ and  without loss of generality that $t < t'$. Then there must be a $t$ somewhere in the last row of $D'$  and we can again  swap $t$ and $t'$ so that the columns are now equal. This completes the proof. 
\end{proof}

\section{Matching field polytopes}

From a  $k \times n$ matching field we can define a  polytope in $\mathbb{R}^{n \times k}$. We expect these polytopes to be of interest in geometric combinatorics. Let $e_{i,j}$ denote coordinates on $\mathbb{R}^{n \times k}$. Given a matching field $\Lambda$, for each $I \in \mathbf{I}_{k,n}$ we set $v_{I, \Lambda} := \sum_{i \in I} e_{i, \Lambda(I)(i)}$. 

\begin{Definition}
Given a   $k \times n$ matching field $\Lambda$ the matching field polytope $\Pi_ \Lambda$ is the convex hull of the set of points 
$\{ v_{I, \Lambda} \ | \ I \in \mathbf{I}_{k,n}  \} $  in $\mathbb{R}^{n \times k}$. 
\end{Definition}

\begin{Proposition}
If $\Lambda$ is a coherent matching field then $\Pi_{\Lambda}$ is the polytope of the toric variety defined by the binomial ideal $\mathcal{I}_{\Lambda}$. 
\end{Proposition}

\begin{Corollary}
Let $\Lambda$ be  a coherent  $k \times n$ matching field then 
$$\frac{1}{[k(n-k)]!}\vol(\Pi_{\Lambda}) \leq  \deg \Gr(k, n).$$ 
\end{Corollary}

Recall that the degree of the Grassmannian is given by the number of standard Young tableaux of shape $\lambda = (\lambda_1, \dots, \lambda_k)$ with $\lambda_i = n-k$ for all $i$. The number of standard Young tableaux is given by the hook-length formula 
$$
 \deg \Gr(k, n) =\frac{{({k}(n -{k}))!} \prod_{1 \le l \le {k}-1} l ! }{\prod_{1 \le l \le {k}} (n - l )! }.
$$
\begin{Example}
Let $\Lambda$ be the hexagonal matching field from Example \ref{exam:hexagon}. 
The matching field polytope $\Pi_{\Lambda}$ has Euclidean volume equal to $\frac{19}{181440}$ and normalised lattice volume equal to $38$.   Whereas the polytope of the diagonal matching field from Example \ref{exam:diag} has volume $\frac{1}{8640}$ and normalised lattice volume equal to $42$. 
The degree of the Grassmannian $(3, 6)$ under the Pl\"ucker embedding is $42$. 
Figure \ref{fig:tableofpolytopes} 
lists all of the cones of the tropical Grassmannian coming from matching fields.

\begin{figure}
\centering
 \begin{tabular}{|c ||c c c c  |} 
 \hline
Bounded $2$-cell  &  Cone of Gr(3,6) & $f$-vector of $\Pi_ \Lambda$ &  $\mathbb{Z}$-vol &
\\ [0.5ex] 
 \hline\hline
$\emptyset$  &  EEEE&  {does not arise from a matching field} & 42  & 
\\ 
\hline
 Triangle    & EEEG  & (20, 123, 386, 728, 882, 700, 358, 111, 18) & 42 &  
 \\ 
 \hline
Diagonal & EEFF(a)  & (20, 122, 372, 670, 766, 571, 276, 83, 14) & 42 &
\\
 \hline
 Parallelogram &  EEFF(b) & (20, 122, 376, 690, 807, 615, 302, 91, 15) & 42 & 
\\ 
\hline 
4-gon  &EEFG  &  (20, 122, 378, 701, 832, 645, 322, 98, 16)   & 42 & 
 \\ 
 \hline
Pentagon & EFFG &(20, 122, 376, 690, 807, 615, 302, 91, 15)  & 42 &
 \\
 \hline
Hexagon & FFFGG &
(20, 120, 361, 641, 720, 526, 250, 75, 13) & 38   &
 \\
 \hline
\end{tabular}
\caption{The $f$-vectors of the polytopes of the different possible initial  degenerations of the Pl\"ucker embedding of the Grassmannian $\Gr(3, 6)$. The description via the bounded $2$-dimensional cell in the tropical hyperplane arrangement is from the classification in \cite{herrmann2009draw}. The first row is a toric degeneration which does not arise from a monomial degeneration of the Pl\"ucker forms, hence it does not come from a matching field.   } \label{fig:tableofpolytopes} 
\end{figure}
\end{Example}

Four of the matching fields in the table  in Figure \ref{fig:tableofpolytopes} arise as $2$-block diagonal matching fields. Namely, the toric degeneration named ``diagonal" comes from diagonal matching field (as well as the isomorphic block diagonal matching field  $B \Lambda_{5,1}$). The toric degeneration named ``parallelogram" comes from the block diagonal matching field $B\Lambda_{1, 5}$. The ``$4-$gon" comes from the matching field $B \Lambda_{4,2}$ and the ``pentagon" comes from $B \Lambda_{2, 4}$. The other rows do not arise from block diagonal matching fields. 

\begin{Remark}We say that two matching fields $\Lambda$ and $\Lambda'$ are \emph{isomorphic} if there exists an element $S_k \times S_n$ sending one to the other.  
In Figure \ref{fig:tableofpolytopes}, the toric degenerations of $\Gr(3, 6)$ from the tropical hyperplane arrangements with bounded cells a parallelogram and a pentagon produce 
isomorphic toric varieties. Already from the table we see that the corresponding polytopes have the same $f$-vector. However, it can  be verified that the  matching fields are
not isomorphic. Therefore, the isomorphism type of the  toric variety of a matching field does not determine the matching field. 
Also, the toric degeneration coming from the diagonal matching field is isomorphic to the  one obtained from  the non-isomorphic $2$-block diagonal matching field  $B \Lambda_{3,3}$. 
\end{Remark}

\bibliographystyle{alpha}
\bibliography{Trop.bib}

\bigskip \bigskip
\bigskip

\noindent
\footnotesize {\bf Authors' addresses:}

\bigskip 

\noindent Fatemeh Mohammadi, University of Bristol, 
BS8 1TW, Bristol, UK
\\ E-mail address: {\tt fatemeh.mohammadi@bristol.ac.uk}

\bigskip 

\noindent Kristin Shaw, University of Oslo, P.O. box 1053, Blindern, 0316 OSLO, Norway
\\ E-mail address: {\tt krisshaw@math.uio.no}
\end{document}